\documentclass[twoside,leqno,10pt, A4]{amsart}
\usepackage{amsfonts}
\usepackage{amsmath}
\usepackage{amscd}
\usepackage{amssymb}
\usepackage{amsthm}
\usepackage{amsrefs}
\usepackage{latexsym}
\usepackage{mathrsfs}
\usepackage{bbm}
\usepackage{enumerate}
\usepackage{graphicx}

\usepackage{amsfonts}
\usepackage{amsmath}
\usepackage{amscd}
\usepackage{amssymb}
\usepackage{amsthm}
\usepackage{amsrefs}
\usepackage{latexsym}
\usepackage{mathrsfs}
\usepackage{bbm}
\usepackage{amscd}
\usepackage{amssymb}
\usepackage{amsthm}
\usepackage{amsrefs}
\usepackage{latexsym}
\usepackage{mathrsfs}
\usepackage{bbm}
\usepackage{enumerate}
\usepackage{graphicx}
\usepackage{color}
\begin{document}

\newtheorem{theorem}[subsection]{Theorem}
\newtheorem{proposition}[subsection]{Proposition}
\newtheorem{lemma}[subsection]{Lemma}
\newtheorem{corollary}[subsection]{Corollary}
\newtheorem{conjecture}[subsection]{Conjecture}
\newtheorem{prop}[subsection]{Proposition}
\numberwithin{equation}{section}
\newcommand{\mr}{\ensuremath{\mathbb R}}
\newcommand{\mc}{\ensuremath{\mathbb C}}
\newcommand{\dif}{\mathrm{d}}
\newcommand{\intz}{\mathbb{Z}}
\newcommand{\ratq}{\mathbb{Q}}
\newcommand{\natn}{\mathbb{N}}
\newcommand{\comc}{\mathbb{C}}
\newcommand{\rear}{\mathbb{R}}
\newcommand{\prip}{\mathbb{P}}
\newcommand{\uph}{\mathbb{H}}
\newcommand{\fief}{\mathbb{F}}
\newcommand{\majorarc}{\mathfrak{M}}
\newcommand{\minorarc}{\mathfrak{m}}
\newcommand{\sings}{\mathfrak{S}}
\newcommand{\fA}{\ensuremath{\mathfrak A}}
\newcommand{\mn}{\ensuremath{\mathbb N}}
\newcommand{\mq}{\ensuremath{\mathbb Q}}
\newcommand{\half}{\tfrac{1}{2}}
\newcommand{\f}{f\times \chi}
\newcommand{\summ}{\mathop{{\sum}^{\star}}}
\newcommand{\chiq}{\chi \bmod q}
\newcommand{\chidb}{\chi \bmod db}
\newcommand{\chid}{\chi \bmod d}
\newcommand{\sym}{\text{sym}^2}
\newcommand{\hhalf}{\tfrac{1}{2}}
\newcommand{\sumstar}{\sideset{}{^*}\sum}
\newcommand{\sumprime}{\sideset{}{'}\sum}
\newcommand{\sumprimeprime}{\sideset{}{''}\sum}
\newcommand{\sumflat}{\sideset{}{^\flat}\sum}
\newcommand{\shortmod}{\ensuremath{\negthickspace \negthickspace \negthickspace \pmod}}
\newcommand{\V}{V\left(\frac{nm}{q^2}\right)}
\newcommand{\sumi}{\mathop{{\sum}^{\dagger}}}
\newcommand{\mz}{\ensuremath{\mathbb Z}}
\newcommand{\leg}[2]{\left(\frac{#1}{#2}\right)}
\newcommand{\muK}{\mu_{\omega}}
\newcommand{\thalf}{\tfrac12}
\newcommand{\lp}{\left(}
\newcommand{\rp}{\right)}
\newcommand{\Lam}{\Lambda_{[i]}}
\newcommand{\lam}{\lambda}
\def\L{\fracwithdelims}
\def\om{\omega}
\def\pbar{\overline{\psi}}
\def\phis{\varphi^*}
\def\lam{\lambda}
\def\lbar{\overline{\lambda}}
\newcommand\Sum{\Cal S}
\def\Lam{\Lambda}
\newcommand{\sumtt}{\underset{(d,2)=1}{{\sum}^*}}
\newcommand{\sumt}{\underset{(d,2)=1}{\sum \nolimits^{*}} \widetilde w\left( \frac dX \right) }

\newcommand{\hf}{\tfrac{1}{2}}
\newcommand{\af}{\mathfrak{a}}
\newcommand{\Wf}{\mathcal{W}}

\theoremstyle{plain}
\newtheorem{conj}{Conjecture}
\newtheorem{remark}[subsection]{Remark}

\makeatletter
\def\widebreve{\mathpalette\wide@breve}
\def\wide@breve#1#2{\sbox\z@{$#1#2$}%
     \mathop{\vbox{\m@th\ialign{##\crcr
\kern0.08em\brevefill#1{0.8\wd\z@}\crcr\noalign{\nointerlineskip}%
                    $\hss#1#2\hss$\crcr}}}\limits}
\def\brevefill#1#2{$\m@th\sbox\tw@{$#1($}%
  \hss\resizebox{#2}{\wd\tw@}{\rotatebox[origin=c]{90}{\upshape(}}\hss$}
\makeatletter

\title[Moments of Dirichlet $L$-functions over Function Fields]{Moments of Dirichlet $L$-functions to a fixed modulus over function fields}

\author[P. Gao]{Peng Gao}
\address{School of Mathematical Sciences, Beihang University, Beijing 100191, China}
\email{penggao@buaa.edu.cn}

\author[L. Zhao]{Liangyi Zhao}
\address{School of Mathematics and Statistics, University of New South Wales, Sydney NSW 2052, Australia}
\email{l.zhao@unsw.edu.au}

\begin{abstract}
 In this paper, we establish the expected order of magnitude of the $k$th-moment of central values of the family of Dirichlet $L$-functions to a fixed prime modulus over function fields for all real $k \geq 0$.
\end{abstract}

\maketitle

\noindent {\bf Mathematics Subject Classification (2010)}: 11M38, 11R59, 11T06   \newline

\noindent {\bf Keywords}: Dirichlet $L$-functions, function fields, lower bounds, moments, upper bounds

\section{Introduction}
\label{sec 1}

  Moments of families of $L$-functions have important arithmetic applications, such as the study of the non-vanishing property of $L$-functions at the central point. In the classical setting, the following $2k$-th moment of central values of the family of Dirichlet $L$-functions to a fixed modulus $q$ has been extensively studied,
\begin{align}
\label{moments}
 \summ_{\substack{ \chi \shortmod q }}|L(\tfrac{1}{2},\chi)|^{2k}.
\end{align}
  Here $k \geq 0$,  $\sum^{\star}$ denotes the sum over primitive Dirichlet characters modulo $q$ and we assume that $q \not \equiv 2 \pmod 4$ to ensure that primitive Dirichlet characters modulo $q$ exist. \newline

The cases $k=1$ and $k=2$ in \eqref{moments} satisfy asymptotic formulas, as evaluated by R. E. A. C. Paley \cite{Paley} and  D. R. Heath-Brown \cite{HB81}, respectively. The result in \cite{HB81} is valid for almost all $q$ and is extended by K. Soundararajan \cite{Sound2007} for all $q$. In \cite{Young2011}, M. P. Young further improved the result in \cite{Sound2007} with a power saving error term for $q$ primes.  Subsequent work in this direction can be found in \cite{BFKMM1, BFKMM, Wu2020}. \newline

   Conjectured formulas concerning \eqref{moments} are given in \cites{CFKRS, BK07, KS00, ConreyFarmer, Keating&Snaith2000} for all $k \geq 0$.  Sharp lower and upper bounds of the conjectured order of magnitude concerning these moments for various values of $k$ can be found in \cite{HB2010, R&Sound1, C&L, BPRZ}. We only point out here that a result of K. Soundararajan \cite{Sound2009} and its refinement by A. J. Harper  \cite{Harper} establish sharp upper bounds for all $k \geq 0$ under the assumption of the generalized Riemann hypothesis (GRH). A modification of a method of M. Radziwi{\l\l} and K. Soundararajan \cite{Radziwill&Sound1} can be applied to establish sharp lower bounds for all $k \geq 1$.  Using a lower bound principle developed by  W. Heap and K. Soundararajan \cite{H&Sound}, P. Gao \cite{Gao2021-4} obtained sharp lower bounds for all $k \geq 0$. \newline

   The aim of this paper is to study the function field analogue of the above family of $L$-functions.  To this end, we fix a finite field $\mathbb{F}_{q}$ of cardinality $q$ and we write $A=\mathbb{F}_{q}[T]$ for the polynomial ring over $\mathbb{F}_{q}$.  Throughout the paper, we reserve the symbol $P$ for a monic, irreducible polynomial in $A$ and we refer to $P$ as a prime in $A$.  We also use the convention that when considering a sum over some subset $S$ of $A$, the symbol $\sum_{f \in S}$ stands for a sum over monic $f \in S$, unless otherwise specified. For any $f \in A$, we write $d(f)$ for its degree and define the norm $|f|$ to be $|f|=q^{d(f)}$ for $f\neq 0$ and $|f|=0$ for $f=0$.  We fix a polynomial $Q \in A$ of degree larger than $1$.  Let $\chi$ be a Dirichlet character modulo $Q$ defined in Section~\ref{sec 2} and $L(s,\chi)$ the $L$-function associated to $\chi$. We are interested in the family of $L$-functions as $\chi$ varies over all primitive characters modulo $Q$. The $2k$-th moment of this family at the central point is conjectured by N. Tamam \cite{Tamam} to satisfy the asymptotic formula
\begin{align}
\label{momentsFF}
 \sumstar_{\substack{ \chi \shortmod Q }}|L(\tfrac{1}{2},\chi)|^{2k} \sim C_k \phis(Q)(\log_q Q)^{k^2},
\end{align}
  where $k \geq 0$,  $\sum^*$ denotes the sum over primitive Dirichlet characters modulo $Q$, $\phis(Q)$ denotes the number of primitive characters modulo $Q$, and $C_k$ is an explicit constant. \newline

  In \cite{Tamam}, Tamam proved that \eqref{momentsFF} is valid for $k=1,2$ by evaluating the second and fourth moments asymptotically for primes $Q$. The result for the fourth moment is extended by J. C. Andrade and M. Yiasemides \cite{AY} to hold for a general polynomial $Q$. In \cite{AY21}, Andrade and Yiasemides further studied mixed fourth moments of all derivatives of the $L$-functions under consideration at the central point.   The sixth power moment of Dirichlet $L$-functions over rational function fields was studied by G. Djankovi\'c and D. \DJ oki\'c \cite{djdo}. \newline

  It is our aim in this paper to establish the $2k$-th moment given in \eqref{momentsFF} to the desired order of magnitude. Our main result is as follows.
\begin{theorem}
\label{thmorderofmag}
   For prime $Q \in A$ such that $|Q|$ is large and any real number $k \geq 0$, we have
\begin{align}
\label{orderofmag}
   \sumstar_{\substack{ \chi \shortmod Q }}|L(\tfrac{1}{2},\chi)|^{2k} \asymp \phis(Q)(\log_q |Q|)^{k^2}.
\end{align}
\end{theorem}

Theorem \ref{thmorderofmag} is proved by establishing sharp lower and upper bounds for the moments, i.e. the two propositions below.
\begin{proposition}
\label{thmlowerbound}
   For prime $Q \in A$ such that $|Q|$ is large and any real number $k \geq 0$, we have
\begin{align}
\label{lowerbound}
   \sumstar_{\substack{ \chi \shortmod Q }}|L(\tfrac{1}{2},\chi)|^{2k} \gg_k \phis(Q)(\log_q |Q|)^{k^2}.
\end{align}
\end{proposition}

Our Proposition~\ref{thmlowerbound} improves upon \cite[Theorem 1.3]{Tamam}, where \eqref{lowerbound} is established for all natural numbers $k$.  Next, the following result gives the upper bound in \eqref{orderofmag}.
\begin{proposition}
\label{thmupperbound}
Using the same notations as in Proposition~\ref{thmlowerbound}, we have
\begin{align}
\label{upperbound}
   \sumstar_{\substack{ \chi \shortmod Q }}|L(\tfrac{1}{2},\chi)|^{2k} \ll_k \phis(Q)(\log_q |Q|)^{k^2}.
\end{align}
\end{proposition}

These propositions will be proved using different approaches.  For lower bound in Proposition~\ref{thmlowerbound}, we will apply the lower bounds principle of Heap-Soundararajan \cite{H&Sound}.  For the upper bounds, we will use the method of Soundararajan \cite{Sound2009} and its refinement by Harper \cite{Harper}.  Note that although this method requires GRH in general, our result is unconditional since GRH has been established in the function field setting.

\section{Preliminaries}
\label{sec 2}

\subsection{Backgrounds on function fields}
\label{sec 2.1}

In this section, we cite some basic facts concerning function fields, most of which can be found in \cite{Rosen02}. Recall that $A=\mathbb{F}_{q}[T]$.  Let $\mathcal{M}$ denote the set of monic polynomials in $A$, $\mathcal{M}_n$ the set of monic polynomials of degree $n$ in $A$ and $\mathcal{M}_{\leq n}$ the set of monic polynomials of degrees not exceeding $n$. Recall further that $P$ denotes a prime in $A$, i. e. $P$ stands for a monic and irreducible element of $A$. \newline

  We define the zeta function $\zeta_A(s)$ associated to $A$ for $\Re(s)>1$ by
\begin{equation*}
\zeta_A(s)=\sum_{\substack{f\in A}}\frac{1}{|f|^{s}}=\prod_{P}(1-|P|^{-s})^{-1},
\end{equation*}
  where we recall our convention that the sum over $f$ is restricted to monic $f \in A$.  Since there are $q^n$ monic polynomials of degree
$n$, it follows that
\begin{equation*}
\zeta_A(s)=\frac{1}{1-q^{1-s}}.
\end{equation*}
  The above expression then defines $\zeta_A(s)$ on the entire complex plane with a simple pole at $s = 1$. We often write
$\zeta_A(s)=\mathcal{Z}(u)$ via a change of variables $u=q^{-s}$, yielding
\begin{equation*}
\mathcal{Z}(u)=\prod_{P}(1-u^{d(P)})^{-1}=(1-qu)^{-1}.
\end{equation*}

  We define a Dirichlet character $\chi$ modulo $f \in A$ in a similar way as that of the analogous object of a number field.  More specifically, let $\chi$ be a homomorphism from $(A/fA)^*$ to $\mc$ and we enlarge its domain to $A/fA$ by defining $\chi(\overline{g})=0$ for any $(g, f) \neq1$, where $\overline{g}$ is the coset to which $g$ belongs in $A/fA$.  We further extend $\chi$ to be defined on $A$ by setting $\chi(g) = \chi(\overline{g})$ for all $g \in A$. Throughout the paper, we shall always regard $\chi$ as a function defined on $A$ instead of on $(A/fA)^*$. For any fixed modulus $f \in A$, $\chi_0$ stands for the principal character modulo $f$ so that $\chi_0(g)=1$ for any $(g, f)=1$. We say a character $\chi$ modulo $f$ is primitive if it cannot be factored through $(A/f'A)^*$ for any proper divisor $f'$ of $f$. In particular, for any prime $Q$, any character $\chi \neq \chi_0 \pmod Q$ is primitive and the total number $\phis(Q)$ of distinct such primitive characters equals $\varphi(Q)-1=|Q|-2$, writing $\varphi$ for the Euler totient function on $A$. \newline

   We define the $L$-function associated to $\chi$ for $\Re(s)>1$ to be
\begin{equation*}
L(s,\chi)=\sum_{\substack{f\in A}}\frac{\chi(f)}{|f|^{s}}=\prod_{P}(1-\chi(P)|P|^{-s})^{-1}.
\end{equation*}

  Similar to the function $\mathcal{Z}(u)$, we have via the change of variables $u=q^{-s}$,
$$ \mathcal{L}(u,\chi) = \sum_{f \in A} \chi(f) u^{d(f)} = \prod_P (1-\chi(P) u^{d(P)})^{-1}.$$

  We also define the von Mangoldt function as
$$ \Lambda(f) =
\begin{cases}
d(P) & \mbox{ if } f=c P^k, c \in \mathbb{F}_q^{\times}, \\
0 & \mbox{ otherwise.}
\end{cases}
$$

\subsection{Preliminary Lemmas}

In this section we include some useful results needed in our proof of Theorem \ref{thmorderofmag}.
We first present a result concerning primes.
\begin{lemma}
\label{RS}
  Denote $\pi(n)$ for the number of primes of degree $n$. We have
\begin{equation}
\label{ppt}
\pi(n) = \frac{q^n}{n}+O \Big(  \frac{q^{n/2}}{n} \Big).
\end{equation}
   For $x \geq 2$ and some constant $b$, we have
\begin{align}
\label{logp}
\sum_{|P| \le x} \frac {\log |P|}{|P|} =& \log x + O(1) \quad \mbox{and} \\
\label{lam2p}
\sum_{|P| \le x} \frac{1}{|P|} =& \log \log x + b+ O\left( \frac {1}{\log x} \right).
\end{align}
  Moreover, for any $\chi \neq \chi_0$ modulo $Q$ and any $z \geq 1$, we have
\begin{align}
\label{wsum}
\sum_{\substack{|P| \leq z }}(\log_q |P|)\chi(P) \ll z^{1/2}.
\end{align}
\end{lemma}
\begin{proof}
  The formulas \eqref{ppt}--\eqref{lam2p} can be found in \cite[Lemma 2.2]{G&Zhao12}.  Hence it remains only to establish \eqref{wsum}. For this, we note that 
\begin{align}
\label{sumlambda}
 \sum_{\substack{|P| \leq z }}(\log_q |P|)\chi(P) = \sum_{ n \leq \log z/\log q}\sum_{|P| =  q^n} (\log_q |P|)\chi(P)=
  \sum_{ n \leq \log z/\log q}n \sum_{|P| =  q^{n}} \chi(P).
\end{align}
  We now combine \cite[Chap. 4, (4)]{Rosen02} and \cite[Chap. 4, (5)]{Rosen02} to see that 
\begin{align}
\label{sumlambda1}
 \sum_{|P| =  q^n} \chi(P) =O(\frac {q^{n/2}}{n}).
\end{align}

  It follows from \eqref{sumlambda} and \eqref{sumlambda1} that 
\begin{align}
\label{ppowerest}
 \sum_{\substack{|P| \leq z }}(\log_q |P|)\chi(P) \ll \sum_{ n \leq \log z/\log q}q^{n/2} \ll z^{1/2}.
\end{align}
  This establishes \eqref{wsum} and hence completes the proof.
\end{proof}

  We end this section by including the following expressions for $L(1/2, \chi)$ and $|L(1/2, \chi)|^2$.
\begin{lemma}
\label{PropDirpoly}
  Let $\chi$ be a primitive character of modulus $R$. We have
\begin{align}
\label{Lexp}
  L(\tfrac 12, \chi)=& \sum_{|f| < |R|}\frac {\chi(f)}{\sqrt{|f|}} \quad \mbox{and} \\
\label{lsquareapprox}
|L(\half, \chi)|^2 =& 2 \sum_{\substack{f, g \\ |fg|< |R|}} \frac{\chi(f) \overline{\chi}(g)}{\sqrt{|fg|}}+O(|R|^{-1/2+\varepsilon}).
\end{align}
\end{lemma}
\begin{proof}
  The expression in \eqref{Lexp} can be found on \cite[p. 189]{Tamam} and the expression in \eqref{lsquareapprox} follows by combining Lemmas 3.10 and 3.11 in \cite{AY21}.
\end{proof}

\subsection{Bounds for $L$-functions}

  In this section, we present several upper bounds concerning $L(s,\chi)$ for a primitive character $\chi$ modulo $Q$. We note from \cite[Proposition 4.3]{Rosen02} that when $\chi \neq \chi_0$, the function $L(s,\chi)$ is a polynomial in $q^{-s}$ of degree at most
$d(Q)-1$, where we recall that $d(Q)$ is the degree of the modulus $Q$ of $\chi$.  We then proceed as in the proof of \cite[Proposition 4.3]{BFK} by setting $m=d(Q)-1, z=0$ there and make use of the proof of \cite[Theorem 3.3]{AT14} to arrive at the following analogue of \cite[Proposition 4.3]{BFK}.
\begin{proposition}
\label{prop-ub}
 Let $\chi$ be a non-principal primitive character modulo $Q$ and $m=d(Q)-1$.  We have for $h \leq m$,
\begin{align}
\label{logLupperbound}
\log \big| L(\tfrac{1}{2}, \chi) \big| \leq \frac{m}{h} + \frac{1}{h} \Re \bigg(  \sum_{\substack{j \geq 1 \\ d(P^j) \leq h}} \frac{
\chi(P^j) \log q^{h - j \deg(P)}}{|P|^{j\big( 1/2+1/(h \log q) \big)} \log q^j} \bigg).
\end{align}
\end{proposition}

  Observe further that Lemma \ref{RS} implies that the terms on the right-hand side of \eqref{logLupperbound} corresponding to $P^j$ with $j \geq 3$
    contribute $O(1)$.  Also, by \eqref{wsum} and partial summation, we see that for any $z \geq 2$ and $\chi^2 \neq \chi_0$,
\begin{align*}
 \sum_{\substack{ |P| \leq z^{1/2} }}  \frac{\chi(P^2)}{|P|^{1+2/\log z}}  \frac{\log (z/|P|^2)}{\log z} = O(1).
\end{align*}

     We apply the observations in \eqref{logLupperbound} by setting $|Q|=q^{d(Q)}, x=q^h$ there to arrive at the following upper bound, analogous to \cite[Proposition 1]{Harper}, for $\log |L(1/2, \chi)|$.
\begin{lemma}
\label{lem: logLbound1}
Let $|Q|$ be large and $2 \leq x \leq |Q|$. We have  for any non-principal primitive character $\chi$ modulo $Q$, 
\begin{align}
\label{equ:3.3'}
\begin{split}
 & \log  |L(\half, \chi)| \leq \Re \left( \sum_{\substack{  |P| \leq x }} \frac{\chi (P)}{|P|^{1/2+1/\log x}}
 \frac{\log (x/|P|)}{\log x} +
 \sum_{\substack{  |P| \leq x^{1/2} }} \frac{\chi (P^2)}{|P|^{1+2/\log x}}  \frac{\log (x/|P|^2)}{2\log x} \right)
 +\frac{\log |Q|}{\log x} + O(1).
\end{split}
 \end{align}
 Moreover, if $\chi^2 \neq \chi_0$, then we have
\begin{align}
\label{equ:3.3}
\begin{split}
 & \log  |L(\half, \chi)| \leq \Re \left( \sum_{\substack{  |P| \leq x }} \frac{\chi (P)}{|P|^{1/2+1/\log x}}  \frac{\log (x/|P|)}{\log x} \right)+ \frac{\log |Q|}{\log x} + O(1).
\end{split}
 \end{align}
\end{lemma}

In order to deal with the sums over primes in \eqref{equ:3.3'} or \eqref{equ:3.3}, we need the following mean value estimate which is similar to \cite[Lemma 3]{Sound2009}.
\begin{lemma}
\label{lem:2.5}
Let $m$ be a natural number such that $y^m  \leq |Q|$. For any complex numbers $a(P)$ we have
 \begin{align*}
   \sum_{\chi \shortmod Q}  \left|\sum_{\substack{|P| \leq y}}\frac{a(P)\chi(P)}{|P|^{1/2}}\right|^{2m}
  \ll_{\varepsilon} & |Q|m!\left (\sum_{\substack{|P| \leq y}}\frac{|a(P)|^2}{|P|}\right)^m.
 \end{align*}
\end{lemma}
\begin{proof}
  Our proof follows closely the proof of \cite[Lemma 3]{Sound2009}. We expand the $m$-power and get
\begin{align*}
  \left|\sum_{\substack{|P| \leq y}}\frac{a(P)\chi(P)}{|P|^{1/2}}\right|^{2m}= \left| \sum_{|f| \leq y^m}\frac {a_{m,y}(f)\chi(f)}{\sqrt{|f|}} \right|^2,
\end{align*}
  where $a_{m,y}(f) = 0$ unless $f$ is the product of $m$ (not necessarily distinct) primes whose norms are all below $y$. In that case, if we write the prime factorization of $f$ as $f = \prod^r_{i=1}P_i^{\alpha_i}$, then $a_{m,y}(f) = \binom{m}{\alpha_1, \ldots, \alpha_r}\prod^r_{i=1}a(P_i)^{\alpha_i}$. \newline

  Now,
\begin{align}
\label{2kbound}
\sum_{\chi \shortmod Q} & \left|\sum_{\substack{|P| \leq y}}\frac{a(P)\chi(P)}{|P|^{1/2}}\right|^{2m}=\sum_{|f|, |g| \leq y^m}\frac {a_{m,y}(f)\overline{a_{m,y}(g)}}{\sqrt{|fg|}}\sum_{\chi \shortmod Q}\chi(f)\overline{\chi(g)}=\varphi(Q)\sum_{\substack{|f|, |g| \leq y^m \\ f \equiv g \shortmod Q}}\frac {a_{m,y}(f)\overline{a_{m,y}(g)}}{\sqrt{|fg|}},
\end{align}
  where the last expression above follows from the familiar orthogonality relation for characters, that is, for monic $u, \ v \in A$:
\begin{align}
\label{orthrel}
  \sum_{\substack{ \chi \shortmod Q }} \chi(u) \overline{\chi}(v)=\left\{
\begin{array}{ll}
  \varphi(Q), & u \equiv v \pmod Q, \\ \\
   0, & \text{otherwise}.
\end{array}
\right.
\end{align}

  As $y^m \leq  |Q|$, we see that the condition $f \equiv g \pmod Q$ in \eqref{2kbound} implies that $f=g$ since they are both monic. It follows that
\begin{align*}
 \sum_{\chi \pmod Q} & \left|\sum_{\substack{|P| \leq y}}\frac{a(P)\chi(P)}{|P|^{1/2}}\right|^{2m} = \varphi(Q)\sum_{\substack{|f| \leq y^m}}\frac {|a_{m,y}(f)|^2}{|f|}.
\end{align*}
  We further estimate right-hand side expression above following the treatments in \cite[Lemma 3]{Sound2009} to arrive at the desired result.
\end{proof}

In the course of proving Theorem \ref{thmorderofmag}, we need to first establish some weaker upper bounds for moments of the related families of $L$-functions in this section.  Let $\mathcal{N}(V, Q)$ be the number of primitive Dirichlet characters $\chi \bmod Q$ such that $\log|L(1/2, \chi)|\geq V$.  Our estimates require the following upper bounds for $\mathcal{N}(V, Q)$ that is similar to \cite[Theorem]{Sound2009}.
\begin{prop}
\label{propNbound}
 Let $|Q|$ be large. If $10\sqrt{\operatorname{log}\operatorname{log} |Q|}\leq V \leq \log \log |Q|$,  then
 \begin{align*}
  \mathcal{N}(V,Q)\ll \frac {|Q|V}{\sqrt{\log \log |Q|}}\operatorname{exp}\left(-\frac{V^2}{\log \log |Q|}
  \left(1-\frac{4}{\operatorname{log}\operatorname{log}\operatorname{log}|Q|}\right)\right).
 \end{align*}
  If $\log \log |Q|<V \leq \frac 14 \log \log |Q| \cdot \log \log \log |Q|$, we have
\begin{align*}
  \mathcal{N}(V,Q)\ll \frac {|Q|V}{\sqrt{\log \log |Q|}}\operatorname{exp}\left(-\frac{V^2}{\log \log |Q|}
  \left(1-\frac{7V}{2(\operatorname{log}\operatorname{log}|Q|)\operatorname{log}\operatorname{log}\operatorname{log}|Q|}\right)^2\right).
\end{align*}
   If $\frac 14 \log \log |Q| \cdot \log \log \log |Q|<V \leq 6 \log |Q|/\log \log |Q|$, we have
\begin{align*}
  \mathcal{N}(V,Q)\ll |Q|\operatorname{exp}\left(-\frac 1{64}V \log V\right).
\end{align*}
\end{prop}
\begin{proof}
   Our proof follows closely that of \cite[Theorem]{Sound2009}. Since there is at most one primitive character $\chi$ modulo $Q$ such that
$\chi^2 =\chi_0$, we may assume throughout the proof that $\chi^2 \neq \chi_0$. We now set $x=|Q|^{A/V}$ with
\begin{align*}
  A= \left\{
\begin{array}{ll}
 \frac{1}{2} \log\log\log |Q|, & 10\sqrt{\log \log |Q|} \leq V \leq \log \log |Q|, \\ \\
 \frac{1}{2V}  \log \log |Q| \cdot \log \log \log |Q|, & \log \log |Q| < V \leq \frac 14 \log \log |Q| \cdot \log \log \log |Q|,\\ \\
   2 , & V > \frac 14 \log \log |Q| \cdot \log \log \log |Q|.
\end{array}
\right.
\end{align*}

  We further set $z=x^{1/\log \log |Q|}$.   Write $M_1$ for the real part of the sum in \eqref{equ:3.3} truncated to $|P| \leq z$ and $M_2$ the complementary sum over $z < |P| \leq x$.  It then follows from \eqref{equ:3.3} that
\[
 \log  |L( \tfrac{1}{2}, \chi)|
    \leq M_1 + M_2 + \frac {V}{A} + O(1).
\]
Hence if  $\log  |L(1/2, \chi)| \geq V $, then we have either
\[ M_2 \geq \frac{V}{8A} \quad \mbox{or} \quad M_1 \geq V_1 := V \left(1-\frac{5}{4A} \right). \]

Now, we set
\[ \operatorname{meas}(Q;M_1) = \# \{\text{primitive $\chi$ modulo $Q$}: M_1 \geq V_1 \} \quad \mbox{and} \quad \operatorname{meas}(Q;M_2) = \# \left\{ \text{primitive $\chi$ modulo $Q$}: M_2 \geq \frac{V}{8A}  \right\}. \]

 Let $[x]$ denote the largest integer not exceeding $x$.  We take $m = [ V/A ] $ so that $x^m \leq |Q|$. We are then able to apply Lemma \ref{lem:2.5} with this $m$ to deduce, aided by Lemma \ref{RS}, that
\begin{align*}
\begin{split}
\left( \frac{V}{8A} \right)^{2m} & \operatorname{meas}(X;M_2) \leq |Q| m! \Big( \sum_{z<|P| \leq x} \frac {1}{|P|} \Big)^{m} \ll |Q| \Big( m(\log \log |Q|+O(1)) \Big)^{m}.
\end{split}
\end{align*}

This leads to
\begin{equation}
\label{equ:bd-S-2}
\operatorname{meas}(X;M_2) \ll  |Q| \left( \frac {8A}{V} \right)^{2m}\Big( m(\log \log |Q|+O(1)) \Big)^{m} \ll |Q|\exp \left(-\frac {V}{2A}\log V \right).
\end{equation}

Next, we estimate $\operatorname{meas}(X;M_1)$.  We apply Lemma \ref{lem:2.5} again to get that for any $m \leq \log |Q|/\log z= V\log \log |Q|/A$,
\begin{align*}
\begin{split}
V^{2m}_1\operatorname{meas}(X;M_1) \leq  &  |Q| m! \Big( \sum_{|P| <z} \frac {1}{|P|} \Big)^{m} \ll |Q| \sqrt{m} \Big( \frac {m \log \log |Q|}{e} \Big)^{m},
\end{split}
\end{align*}
 where the last estimate above follows from Lemma \ref{RS} and Stirling's formula (see \cite[(5.112)]{iwakow}), which asserts that
\begin{align}
\label{Stirling}
\begin{split}
 m! \ll \sqrt{m} ( \frac {m }{e} )^{m}.
\end{split}
\end{align}

  It follows that
\begin{align*}
\begin{split}
 \operatorname{meas}(X;M_1) \ll |Q| \sqrt{m} \Big( \frac {m \log \log |Q|}{eV^{2}_1} \Big)^{m}.
\end{split}
\end{align*}

If $V \leq (\log \log |Q|)^2$, we take $m=[V^2_1/\log \log |Q|]$.  Otherwise if $V > (\log \log |Q|)^2$, we take $m=[10V]$.  These choices give raise to the bound
\begin{align}
\label{M1bound1}
\operatorname{meas}(X ; M_1)
\ll |Q|\frac {V}{\sqrt{\log \log |Q|}}\operatorname{exp}\left(-\frac{V_1^2}{\log \log |Q|} \right)+|Q|\operatorname{exp}\left(-4V \log V \right).
\end{align}

   Note that
\begin{align*}
 \operatorname{exp}\left(-4V \log V \right) \ll \exp \left( -\frac {V}{2A}\log V \right).
\end{align*}

  Moreover, we have for $V \leq \frac{1}{4} \log \log |Q| \cdot \log \log \log |Q|$,
\begin{align*}
  \exp \left(-\frac {V}{2A}\log V \right) \ll \operatorname{exp}\left(-\frac{V_1^2}{\log \log |Q|} \right).
\end{align*}

   On the other hand, if $V \geq  \frac{1}{4} \log \log |Q| \cdot \log \log \log |Q|$,  $V_1=3V/8$ so that
\begin{align}
\label{M1bound2}
\exp \left(-\frac {V}{2A}\log V\right) = \exp \left(-\frac {V}{4}\log V\right) \; \mbox{and} \; \exp \left(-\frac{V_1^2}{\log \log |Q|} \right) = \exp \left(-\frac{9V^2}{64\log \log |Q|} \right) \ll \operatorname{exp}\left(-\frac{V \log V}{64} \right).
\end{align}

  The assertion of the proposition now follows from \eqref{equ:bd-S-2}, \eqref{M1bound1} and \eqref{M1bound2}.
\end{proof}

  Now, Proposition \ref{propNbound} allows us to establish the following weaker upper bounds for moments of the $L$-functions under our consideration.

\begin{prop}
\label{prop: upperbound}
Let $k$ be a positive integer and $\varepsilon>0$ be given.  We have, for large $|Q|$,
\begin{align*}
    \sumstar_{\chi \shortmod Q}  \left| L \left( \tfrac{1}{2}, \chi \right) \right|^{2k}  \ll_k  |Q|(\log_q |Q|)^{k^2+\varepsilon} .
\end{align*}
\end{prop}
\begin{proof}
  We note that
\begin{align}
\label{momentint}
   \sumstar_{\chi \shortmod Q}  \left| L \left( \tfrac{1}{2}, \chi \right) \right|^{2k}  =& -\int\limits_{-\infty}^{+\infty}\exp (2kV) \dif \mathcal{N}(V,Q)
  = 2k\int\limits_{-\infty}^{+\infty}\operatorname{exp}(2kV)\mathcal{N}(V,Q) \dif V,
\end{align}
after integration by parts.  As $N(V, Q) \ll |Q|$, we see that
\begin{align*}
 2k\int\limits_{-\infty}^{10\sqrt{\log \log |Q|}}\operatorname{exp}(2kV)\mathcal{N}(V,Q) \dif V \ll |Q| \int\limits_{-\infty}^{10\sqrt{\log \log |Q|}}\operatorname{exp}(2kV) \dif V \ll |Q|(\log_q |Q|)^{k^2}.
\end{align*}

  Moreover, by taking $x = \log |Q|$ in (\ref{equ:3.3}) and bounding the sum  over $P$ in (\ref{equ:3.3}) trivially, we see that $\mathcal{N} (V,Q) =0$  for $V > 6\log |Q|/\log \log |Q|$. Thus, it remains to consider the $V$-range with $10\sqrt{\log \log |Q|} \leq V \leq 6\log |Q|/\log \log |Q|$. \newline

Now Proposition~\ref{propNbound} yields that for $10\sqrt{\log \log |Q|} \leq V \leq 6\log |Q|/\log \log |Q|$,
\begin{align}
 \label{equ:rough-01}
  \mathcal{N}(V, X)\ll
\left\{
\begin{array}{ll}
  |Q|(\log_q |Q|)^{o(1)}\operatorname{exp}\left(-\frac{V^2}{\log \log |Q|}\right), & 10\sqrt{\log \log |Q|} \leq V \leq 4k \log \log |Q|, \\ \\
   |Q|(\log_q |Q|)^{o(1)}\operatorname{exp}(-3kV), &  V > 4k \log \log |Q|.
\end{array}
\right.
\end{align}
  Applying the bounds in \eqref{equ:rough-01} to evaluate the integral in \eqref{momentint} now leads to the desired result.
\end{proof}

\section{Proof of Proposition \ref{thmlowerbound}}
\label{sec 2'}

\subsection{Lower bounds principle}

   We may assume that $k \neq 1$ throughout as the case $k=1$ for \eqref{orderofmag} is already established.  Let $N$ and
 $M$ be two large natural numbers depending on $k$ only and $\{ \ell_j \}_{1 \leq j \leq R}$ a sequence of even natural
  numbers defined in the following manner.  Let $\ell_1= 2\lceil N \log \log |Q| \rceil$ and $\ell_{j+1} = 2 \lceil N \log \ell_j \rceil$ for $j \geq 1$, where $R$ is the largest natural number satisfying $\ell_R >10^M$.  We may assume that $M$ is chosen so that we have $\ell_{j} >
  \ell_{j+1}^2$ for all $1 \leq j \leq R-1$ and this further implies that
\begin{align}
\label{sumoverell}
  \sum^R_{j=1}\frac 1{\ell_j} \leq \frac 2{\ell_R}.
\end{align}

    We write ${P}_1$ for the set of primes whose norms not exceeding $|Q|^{1/\ell_1^2}$ and
${ P_j}$ for the set of primes whose norms lie in the interval $(|Q|^{1/\ell_{j-1}^2}, |Q|^{1/\ell_j^2}]$ for $2\le j\le R$. For each $1 \leq j \leq R$ and any real number $\alpha$, set
\begin{equation*}
{\mathcal P}_j(\chi) = \sum_{P\in P_j} \frac{\chi(P)}{\sqrt{|P|}}, \quad {\mathcal N}_j(\chi, \alpha) = E_{\ell_j} (\alpha {\mathcal P}_j(\chi)), \quad \mathcal{N}(\chi, \alpha) = \prod_{j=1}^{R} {\mathcal
N}_j(\chi,\alpha),
\end{equation*}
  where we define, for any real number $\ell>0$ and $x$,
\begin{align}
\label{E}
E_{\ell}(x) = \sum_{j=0}^{[\ell]} \frac{x^{j}}{j!}.
\end{align}

  We now apply the lower bounds principle of W. Heap and K. Soundararajan \cite{H&Sound}.  By H\"older's inequality, we get for $0<k<1$,
\begin{align*}
\begin{split}
 \sumstar_{\substack{ \chi \shortmod Q}} & L(\tfrac{1}{2},\chi)  \mathcal{N}(\chi, k-1) \mathcal{N}(\overline{\chi}, k) \\
\leq & \Big ( \sumstar_{\substack{ \chi \shortmod Q }}|L(\tfrac{1}{2},\chi)|^{2k} \Big )^{1/2}\Big ( \sumstar_{\substack{ \chi \shortmod Q
 }}|L(\tfrac{1}{2},\chi) \mathcal{N}(\chi, k-1)|^2   \Big)^{(1-k)/2}\Big ( \sumstar_{\substack{ \chi \shortmod Q }} |\mathcal{N}(\chi,
 k)|^{2/k}|\mathcal{N}(\chi, k-1)|^{2}  \Big)^{k/2}.
\end{split}
\end{align*}

  Similarly, for $k>1$,
\begin{align*}
\begin{split}
   & \sumstar_{\substack{ \chi \shortmod Q }}L(\tfrac{1}{2},\chi)  \mathcal{N}(\chi, k-1) \mathcal{N}(\overline{\chi}, k)
 \leq  \Big ( \sumstar_{\substack{ \chi \shortmod Q }}|L(\tfrac{1}{2},\chi)|^{2k} \Big )^{1/2k}\Big ( \sumstar_{\substack{ \chi
 \shortmod Q }} |\mathcal{N}(\chi, k)\mathcal{N}(\chi, k-1)|^{2k/(2k-1)}  \Big)^{(2k-1)/(2k)}.
\end{split}
\end{align*}

Hence in order to prove Proposition \ref{thmlowerbound}, it suffices to establish the following three propositions.
\begin{proposition}
\label{Prop4} With notations as above, we have
\begin{align}
\label{Aestmation}
\sumstar_{\substack{ \chi \shortmod Q }}L(\tfrac{1}{2},\chi) \mathcal{N}(\overline{\chi}, k) \mathcal{N}(\chi, k-1) \gg \phis(Q)(\log_q |Q|)^{ k^2
} .
\end{align}
\end{proposition}

\begin{proposition}
\label{Prop5} With notations as above, we have
\begin{align}
\label{Nestmation}
 \sumstar_{\substack{ \chi \shortmod Q }}|L(\tfrac{1}{2},\chi)\mathcal{N}(\chi, k-1)|^2  \ll \phis(Q)(\log_q |Q|)^{ k^2} .
\end{align}
\end{proposition}

\begin{proposition}
\label{Prop6} With notations as above, we have
\begin{align*}
\max \Big ( \sumstar_{\substack{ \chi \shortmod Q }} |\mathcal{N}(\chi,
 k)|^{2/k}|\mathcal{N}(\chi, k-1)|^{2}, \sumstar_{\substack{ \chi
 \shortmod Q }} |\mathcal{N}(\chi, k)\mathcal{N}(\chi, k-1)|^{2k/(2k-1)} \Big )   \ll \phis(Q)(\log_q |Q|)^{ k^2} .
\end{align*}
\end{proposition}

     Our proofs of the above propositions are similar to those for Propositions 3.3--3.5 in \cite{Gao2021-4}. We shall therefore omit the proof of Proposition \ref{Prop6} and be brief on the proofs of Propositions~\ref{Prop4} and~\ref{Prop5}.

\subsection{Proof of Proposition \ref{Prop4}}
\label{sec 4}

   Let $\Omega(f)$ denote the number of distinct prime powers dividing $f$ and $w(f)$ the multiplicative function such that
    $w(P^{\alpha}) = \alpha!$ for prime powers $P^{\alpha}$.  Let $b_j(f), 1 \leq j \leq R$ be functions such that $b_j(f)=1$ when $f$ is
    composed of at most $\ell_j$ primes, all from the interval $P_j$. Otherwise, we define $b_j(f)=0$. We use these notations to see that for
    any real number $\alpha$,
\begin{equation*}
{\mathcal N}_j(\chi, \alpha) = \sum_{f_j} \frac{1}{\sqrt{|f_j|}} \frac{\alpha^{\Omega(f_j)}}{w(f_j)}  b_j(f_j) \chi(f_j), \quad 1\le j\le R.
\end{equation*}
    Each ${\mathcal N}_j(\chi, \alpha)$ is a short Dirichlet polynomial since $b_j(f_j)=0$ unless $|f_j| \leq
    (|Q|^{1/\ell_j^2})^{\ell_j}=|Q|^{1/\ell_j}$. It follows from this that ${\mathcal N}(\chi, k)$ and ${\mathcal N}(\chi, k-1)$ are short Dirichlet
    polynomials whose lengths are both at most $|Q|^{1/\ell_1+ \ldots +1/\ell_R} < |Q|^{2/10^{M}}$ by \eqref{sumoverell}. Moreover, it is readily
    checked that for each $\chi$ modulo $Q$ (including the case $\chi=\chi_0$),
\begin{align}
\label{prodNbound}
 & {\mathcal N}(\overline \chi, k){\mathcal N}(\chi, k-1) \ll |Q|^{2(1/\ell_1+ \ldots +1/\ell_R)} < |Q|^{4/10^{M}}.
\end{align}

  We deduce from the above and Lemma~\ref{PropDirpoly} that
\begin{align*}
\begin{split}
 \sumstar_{\substack{ \chi \shortmod Q }}L(\tfrac{1}{2},\chi) \mathcal{N}(\overline{\chi}, k) \mathcal{N}(\chi, k-1) = & \sumstar_{\substack{ \chi \shortmod Q }}\sum_{|f| <|Q|}\frac {\chi(f)}{\sqrt{|f|}}\mathcal{N}(\overline{\chi}, k) \mathcal{N}(\chi,
k-1) \\
= & \sum_{\substack{ \chi \shortmod Q }}\sum_{|f| <|Q|}\frac {\chi(f)}{\sqrt{|f|}}\mathcal{N}(\overline{\chi}, k) \mathcal{N}(\chi,
 k-1)+O(|Q|^{1/2+4/10^{M}}) \\
=& \varphi(Q) \sum_{a} \sum_{b} \sum_{\substack{|f| < |Q| \\ af \equiv b \bmod Q}}\frac {x_a y_b}{\sqrt{|abf|}}+O(|Q|^{1/2+4/10^{M}}),\
\end{split}
\end{align*}
  where the last estimation above follows from \eqref{prodNbound} and where we write for simplicity
\begin{align*}
 {\mathcal N}(\chi, k-1)= \sum_{|a|  \leq |Q|^{2/10^{M}}} \frac{x_a}{\sqrt{|a|}} \chi(a) \quad \mbox{and} \quad \mathcal{N}(\overline{\chi}, k) = \sum_{|b|  \leq  |Q|^{2/10^{M}}} \frac{y_b}{\sqrt{|b|}}\overline{\chi}(b).
\end{align*}

  We now consider the contribution from the terms $af \neq b$ in the last expression of \eqref{Aestmation}. As $|b| <|Q|$, we see that $af \equiv b \bmod Q$ occurs only when $d(af)>d(b)$ so that we may write $af=b+l Q$ with $l \in A$.  Since $af$ is monic, so is $b+l Q$.  As $d(b+lQ)=d(lQ)$, this implies that $l$ is monic. Note further that $af=b+l Q$ implies that $|l| \leq  |Q|^{2/10^{M}}$, we deduce, together with the observation that $x_a, y_b \ll 1$, that the total contribution from these terms is
\begin{align*}
 \ll & \varphi(Q)  \sum_{|b|  \leq |Q|^{2/10^{M}}}  \sum_{|l| \leq |Q|^{2/10^{M}}}\frac {1}{\sqrt{|blQ|}} \ll |Q|^{1/2+2/10^{M}}.
\end{align*}

  We thus obtain
\begin{align*}
&  \sumstar_{\substack{ \chi \shortmod Q }}L(\tfrac{1}{2},\chi) \mathcal{N}(\overline{\chi}, k) \mathcal{N}(\chi, k-1)
\gg \varphi(Q) \sum_{a} \sum_{b} \sum_{\substack{|f| < |Q| \\ af = b }}\frac {x_a y_b}{\sqrt{|abf|}}
= \varphi(Q) \sum_{b} \frac {y_b}{|b|} \sum_{\substack{a, f \\ af = b }}x_a=\varphi(Q) \sum_{b} \frac {y_b}{|b|} \sum_{\substack{a | b }}x_a,
\end{align*}
  where the last equality above follows from the observation that $b \leq |Q|^{2/10^{M}}<|Q|$. \newline

We then proceed as in the proof of Proposition 3.3 in \cite{Gao2021-4}, getting the desired estimate in \eqref{Aestmation}.

\subsection{Proof of Proposition \ref{Prop5}}
\label{sec 5}

  Recall from  Section \ref{sec 4} that ${\mathcal N}(\chi, k-1)$ is a short Dirichlet polynomial with length not exceeding $|Q|^{2/10^{M}}$. This allows us to write
$$ |\mathcal{N}(\chi, k-1)|^2 =\sum_{|a|,|b| \leq |Q|^{2r_k/10^{M}}} \frac{u_a u_b}{\sqrt{ab}}\chi(a)\overline{\chi}(b),$$
 where $u_a, u_b$ are real numbers satisfying
\begin{align}
\label{uest}
   0 \leq u_a, u_b \leq 1.
\end{align}

 We now apply \eqref{lsquareapprox} to estimate the left-hand side expression in \eqref{Nestmation}. As ${\mathcal N}(\chi, k-1)$ is a short Dirichlet polynomial, this together with \eqref{uest} implies the contribution of the $O$-term in \eqref{lsquareapprox} is negligible.  It follows that
\begin{align}
\label{LNsquaresum}
\begin{split}
\sumstar_{\substack{ \chi \shortmod Q }} |L(\tfrac{1}{2},\chi)|^2|\mathcal{N}(\chi, k-1)|^2
\ll & \sum_{|a|,|b| \leq |Q|^{2r_k/10^{M}}} \frac{u_a u_b}{\sqrt{|ab|}} \sum_{|fg|<|Q|} \frac{1}{\sqrt{|fg|}} \ \sumstar_{\substack{ \chi \shortmod Q }} \chi(af) \overline{\chi}(bg) \\
\ll & \sum_{D | Q}\mu_A(D)\varphi(Q/D) \sum_{|a|,|b| \leq |Q|^{2r_k/10^{M}}} \frac{u_a u_b}{\sqrt{|ab|}} \sum_{\substack{|fg|<|Q| \\ (fg, Q)=1 \\ af \equiv bg \,\shortmod {Q/D}}}
\frac{1}{\sqrt{|fg|}},
\end{split}
\end{align}
  where we denote $\mu_A$ for the M\"obius function on $A$ and the last estimation above follows from a simple orthogonality relation, which asserts that for $(uv, Q)=1$, we have
\begin{align*}
\begin{split}
  \sumstar_{\substack{ \chi \shortmod Q }} \chi(u) \overline{\chi}(v)=\sum_{\substack{D | Q \\ u \equiv v \shortmod {Q/D}}}\mu_A(D)\varphi(Q/D).
\end{split}
\end{align*}

   To estimation of the last display in \eqref{LNsquaresum}, we first notice that the contribution of $D=Q$ in the last display in \eqref{LNsquaresum} is
\begin{align}
\label{LNsquaresumD=Q}
\begin{split}
 \ll  |Q|^{\varepsilon}  \sum_{|a|,|b| \leq |Q|^{2r_k/10^{M}}} & \frac{1}{\sqrt{|ab|}} \sum_{\substack{|fg|<|Q| \\ (fg, Q)=1 }}
\frac{1}{\sqrt{|fg|}} \ll  |Q|^{2r_k/10^{M}+\varepsilon} \sum_{\substack{|fg|<|Q|}}\frac{1}{\sqrt{|fg|}} \\
\ll & |Q|^{2r_k/10^{M}+\varepsilon} \sum_{\substack{|f|<|Q|  }}\frac {\tau_A(f)}{\sqrt{|f|}} \ll |Q|^{1/2+2r_k/10^{M}+\varepsilon},
\end{split}
\end{align}
  where we denote $\tau_A$ for the divisor function on $A$ and the last estimation above follows from the observation that, similar to the integer case given in \cite[Theorem 2.11]{MVa}, for any $\varepsilon >0$,
\begin{align*}
\begin{split}
 \tau_A(f) \ll |f|^{\varepsilon}.
\end{split}
\end{align*}

   We now conclude from  \eqref{LNsquaresum} and \eqref{LNsquaresumD=Q} that the contribution of $D=Q$ to \eqref{LNsquaresum} is negligible and
\begin{align}
\label{LNsquaresum1}
\begin{split}
& \sumstar_{\substack{ \chi \shortmod Q }}|L(\tfrac{1}{2},\chi)|^2|\mathcal{N}(\chi, k-1)|^2
\ll  \varphi(Q) \sum_{|a|,|b| \leq |Q|^{2r_k/10^{M}}} \frac{u_a u_b}{\sqrt{|ab|}} \sum_{\substack{|fg|<|Q| \\ (fg, Q)=1 \\ af \equiv bg \,\shortmod {Q}}}
\frac{1}{\sqrt{|fg|}}.
\end{split}
\end{align}

 We now estimate the contribution of the terms with $af \neq bg$ in \eqref{LNsquaresum1}.  We may assume that $d(af) \geq d(bg)$ without loss of generality and write $af=bg+lQ$ for some $0 \neq l \in A$.  It follows that $d(lQ) \leq d(bg+lQ) \leq d(af)$, so that $|lQ| \leq |af| \leq |Q|^{1+2r_k/10^{M}}$ which implies that $|l| \leq |Q|^{2r_k/10^{M}}$. Moreover, $1/\sqrt{|af|} \ll 1/\sqrt{|lQ|}$ and $|fg| < |Q|$ implies $|abfg|  <  |abQ| \leq |Q|^{1 +4r_k/10^{M}}$, so that we have $|g| \leq |bg| \leq |Q|^{1/2 + 2r_k/10^{M}}$. We then deduce that the contribution from the terms $af \neq bg$ is
\begin{align*}
  \ll  |Q| \sum_{|a|,|b| \leq |Q|^{2r_k/10^{M}}} & \frac{1}{\sqrt{|ab|}} \sum_{\substack{f, g \\ |fg| < |Q| \\ af = bg+lQ \\ |l| \geq 1}} \frac{1}{\sqrt{|fg|}} \\
\ll &  |Q| \sum_{|b| \leq |Q|^{2r_k/10^{M}}} \frac{1}{\sqrt{|b|}} \sum_{\substack{|g|  \leq |Q|^{1/2+2r_k/10^{M}}}} \frac{1}{\sqrt{|g|}}\sum_{\substack{|l| \leq |Q|^{2r_k/10^{M}}}} \frac{1}{\sqrt{|lQ|}} \ll  |Q|^{1-\varepsilon}.
\end{align*}

Thus it remains to consider the terms $af = bg$ in the last expression of \eqref{LNsquaresum}. We write $f = \alpha b/(a,b)$, $g =
 \alpha a/(a,b)$ for some $\alpha \in A$ and these terms in question are
\begin{align}
\label{mainterm}
\begin{split}
\ll & \phis(Q) \sum_{|a|,|b| \leq |Q|^{2r_k/10^{M}}} \frac{|(a,b)|}{|ab|} u_a u_b  \sum_{\substack{ |\alpha| < |Q||(a,b)|^2/|ab| \\ (\alpha, Q)=1}} \frac{1}{|\alpha|} \ll  \phis(Q)\sum_{|a|,|b| \leq |Q|^{2r_k/10^{M}}} \frac{|(a,b)|}{|ab|} u_a u_b  \sum_{\substack{ |\alpha| < |Q||(a,b)|^2/|ab|}} \frac{1}{|\alpha|},
\end{split}
\end{align}
 where the last estimation above follows by observing that $|Q||(a,b)|^2/|ab|<|Q|$. \newline

  To evaluate the last sum in \eqref{mainterm}, we set $X=|Q||(a,b)|^2/|ab|$.  This gives
\begin{align}
\label{alphasum}
& \sum_{\substack{ |\alpha| < X}} \frac{1}{|\alpha|}= \sum^{d(X)-1}_{n=0}q^{-n}q^n=d(X).
\end{align}

Now \eqref{alphasum} renders that \eqref{mainterm} is
\begin{align*}
& \ll \phis(Q) (r_k\ell_{v+1})^2 \Big( \frac{12 r_k }{e } \Big)^{2r_k\ell_{v+1}} \sum_{|a|,|b| \leq |Q|^{2r_k/10^{M}}} \frac{|(a,b)|}{|ab|}  u_a u_b \Big ( \log_q |Q|+2 \log_q |(a, b)|-\log_q |a| -\log_q |b| \Big ).
\end{align*}

We then proceed as in the proof of \cite[Proposition 3.4]{Gao2021-4}, getting the estimate in \eqref{Nestmation} and completing the proof of the proposition.

\section{Proof of Proposition \ref{thmupperbound}}

  Exponentiating both sides of \eqref{equ:3.3'} gives that
\begin{align}
\label{basicest0}
\begin{split}
 & \left| L \left( \tfrac{1}{2}, \chi \right) \right|^{2k}  \ll
 \exp \left(2k \Re \left( \sum_{\substack{  |P| \leq x }} \frac{\chi (P)}{|P|^{1/2+1/\log x}}
 \frac{\log (x/|P|)}{\log x} +
 \sum_{\substack{  |P| \leq x^{1/2} }} \frac{\chi (P^2)}{|P|^{1+2/\log x}}  \frac{\log (x/|P|^2)}{2\log x}
 +\frac{\log |Q| }{\log x}\right) \right ).
\end{split}
 \end{align}
 
  Upon setting $x=\log \log |Q|$ in \eqref{basicest0} and estimating the right-hand expression trivially, we see that $\left| L \left( \tfrac{1}{2}, \chi \right) \right|^{2k} \ll Q$. As the right-hand side expression of \eqref{upperbound} is easily seen to be $ \gg Q$ and there is at most one primitive character $\chi$ modulo $Q$ such that $\chi^2 =\chi_0$, we deduce from Lemma \ref{lem: logLbound1} that we may assume that the estimation given in \eqref{equ:3.3} is satisfied by all $\chi$. Thus, we obtain upon exponentiating both sides of \eqref{equ:3.3} that
\begin{align}
\label{basicest}
\begin{split}
 & \left| L \left( \tfrac{1}{2}, \chi \right) \right|^{2k}  \ll
 \exp \left(2k \Re \left( \sum_{\substack{  |P| \leq x }} \frac{\chi (P)}{|P|^{1/2+1/\log x}}
 \frac{\log (x/|P|)}{\log x} 
 +\frac{\log |Q| }{\log x}\right) \right ).
\end{split}
 \end{align} 

Following the approach by A. J. Harper \cite{Harper}, we define, for a large number $T$,
$$ \alpha_{0} = \frac{\log 2}{\log |Q| }, \;\;\;\;\; \alpha_{i} = \frac{20^{i-1}}{(\log\log |Q| )^{2}} \;\;\; \mbox{for} \; i \geq 1 \quad \mbox{and} \quad
\mathcal{J} = \mathcal{J}_{k, Q } = 1 + \max\left\{i : \alpha_{i} \leq 10^{-T} \right\} . $$

   We shall set $x=|Q| ^{\alpha_j}$ for $j \geq 1$ in \eqref{basicest} in what follow and we set
\[ {\mathcal M}_{i,j}(\chi) = \sum_{|Q| ^{\alpha_{i-1}} < |P| \leq |Q| ^{\alpha_{i}}}  \frac{\chi (P)}{|P|^{1/2+1/(\log |Q| ^{\alpha_{j}})}} \frac{\log (|Q| ^{\alpha_{j}}/|P|)}{\log |Q| ^{\alpha_{j}}}, \quad 1\leq i \leq j \leq \mathcal{J}. \]

 We also define for $1\leq j \leq \mathcal{J}$,
\begin{align*}
 \mathcal{S}(j) =& \left\{ \text{primitive  } \chi \shortmod Q : | \Re {\mathcal M}_{i,l}(\chi)| \leq \alpha_{i}^{-3/4} \; \; \mbox{for all}  \; 1 \leq i \leq j, \; \mbox{and} \; i \leq l \leq \mathcal{J}, \right. \\
 & \hspace*{6cm} \left. \;\;\;\;\; \text{but }  |\Re{\mathcal M}_{j+1,l}(\chi)| > \alpha_{j+1}^{-3/4} \; \text{ for some } j+1 \leq l \leq \mathcal{J} \right\} . \\
 \mathcal{S}(\mathcal{J}) =& \left\{ \text{primitive  } \chi \shortmod Q : |\Re{\mathcal M}_{i, \mathcal{J}}(\chi)| \leq \alpha_{i}^{-3/4} \; \mbox{for all}  \; 1 \leq i \leq \mathcal{J} \right\}.
\end{align*}

  We first note that,
\begin{align*}
\begin{split}
\text{meas}(\mathcal{S}(0)) \leq & \sum_{\substack{\chi \shortmod Q}} \sum^{\mathcal{J}}_{l=1}
\Big ( \alpha^{3/4}_{1}{|\Re \mathcal
M}_{1, l}(\chi)| \Big)^{2\lceil 1/(10\alpha_{1})\rceil } \leq \sum^{\mathcal{J}}_{l=1}\sum_{\substack{\chi \shortmod Q}} \Big ( \alpha^{3/4}_{1}{|\mathcal
M}_{1, l}(\chi)| \Big)^{2\lceil 1/(10\alpha_{1})\rceil }.
\end{split}
\end{align*}

  We apply Lemma \ref{lem:2.5} to bound the last expression above to see that 
\begin{align}
\label{Smest}
\begin{split}
 \text{meas}(\mathcal{S}(0)) \ll &  \mathcal{J} |Q| (\lceil 1/(10\alpha_{1})\rceil !) (\alpha^{3/4}_{1})^{2 \lceil 1/(10\alpha_{1})\rceil}\Big (\sum_{|P|  \leq |Q|^{\alpha_1}} \frac{1}{|P|}\Big )^{ \lceil 1/(10\alpha_{1})\rceil}  \\
\ll & \mathcal{J} |Q| \sqrt{\lceil 1/(10\alpha_{1})\rceil }\big(\frac {\lceil 1/(10\alpha_{1})\rceil }{e}\big)^{ \lceil 1/(10\alpha_{1})\rceil} (\alpha^{3/4}_{1})^{2 \lceil 1/(10\alpha_{1})\rceil}\Big (\sum_{|P|  \leq |Q|^{\alpha_1}} \frac{1}{|P|}\Big )^{ \lceil 1/(10\alpha_{1})\rceil}.
\end{split}
 \end{align}

   Now Lemma \ref{RS} gives
\begin{align*}
 \mathcal{J} \leq \log\log\log |Q|  , \quad \alpha_{1} = \frac{1}{(\log\log |Q|)^{2}}  \quad \mbox{and} \quad \sum_{|P|  \leq |Q|^{1/(\log\log |Q|)^{2}}} \frac{1}{|P|} \leq \log\log |Q|=\alpha^{-1/2}_1 .
\end{align*}

Applying these estimates to \eqref{Smest} yields
\begin{align*}
\text{meas}(\mathcal{S}(0)) \ll &
\mathcal{J}|Q| \sqrt{\lceil 1/(10\alpha_{1})\rceil } e^{-1/(10\alpha_{1})}\ll |Q| e^{-(\log\log |Q|)^{2}/20}  .
\end{align*}
  We then deduce via the Cauchy-Schwarz inequality and Proposition \ref{prop: upperbound}  that
\begin{align}
\label{LS0bound}
\begin{split}
\sum_{\chi \in  \mathcal{S}(0)} \left| L \left( \tfrac{1}{2}, \chi \right) \right|^{2k}  \leq &   \left( \text{meas}(\mathcal{S}(0)) \cdot
\sumstar_{\substack{\chi \shortmod Q}} |L(\half, \chi)|^{4k} \right)^{1/2}
 \\
 \ll & \left( |Q|  \exp\left( -(\log\log |Q| )^{2}/20 \right) |Q|  (\log_q |Q| )^{(2k)^{2}+1} \right)^{1/2} \ll  |Q| (\log_q |Q| )^{k^2}.
\end{split}
\end{align}

  Notice that $\{ \text{primitive  } \chi \pmod Q\}  = \bigcup_{j=0}^{ \mathcal{J}}\mathcal{S}(j)$, so that we deduce from this and \eqref{LS0bound} that it suffices to show that
\begin{align}
\label{sumovermj}
  \sum_{j=1}^{\mathcal{J}}\sum_{\chi \in \mathcal{S}(j)} |L( \tfrac{1}{2}, \chi)|^{2k}
   \ll |Q| (\log_q |Q|)^{k^{2}} .
\end{align}

    Now,  we fixing a $j$ with $1 \leq j \leq \mathcal{J}$ and set $x=|Q| ^{\alpha_j}$ in \eqref{basicest} to arrive at
\begin{align*}
\begin{split}
 & \left| L \left( \tfrac{1}{2}, \chi \right) \right|^{2k} \ll \exp \left(\frac {2k}{\alpha_j} \right) \exp \Big (
 2k\Re\sum^j_{i=1}{\mathcal M}_{i,j}(\chi) \Big ).
\end{split}
 \end{align*}

  As we have $|\Re{\mathcal M}_{i, j}| \leq  \alpha^{-3/4}_i$ when $\chi \in \mathcal{S}(j)$, we can directly apply \cite[Lemma 5.2]{Kirila} to obtain that
\begin{align*}
\begin{split}
\exp \Big ( 2k \Re\sum^j_{i=1}{\mathcal M}_{i,j}(\chi)\Big ) \ll
\prod^j_{i=1}E_{e^2k\alpha^{-3/4}_i}(k\Re{\mathcal M}_{i,j}(\chi))^2.
\end{split}
 \end{align*}

   We then deduce from the description on $\mathcal{S}(j)$ that when $j \geq 1$,
\begin{align*}
\begin{split}
  \sum_{\chi \in \mathcal{S}(j)} & \left| L \left( \tfrac{1}{2}, \chi \right) \right|^{2k}  
 \ll  \exp \left( \frac {4k}{\alpha_j} \right)
 \sum^{\mathcal{J}}_{l=j+1} \sum_{\chi \in \mathcal{S}(j)}
\prod^j_{i=1}E_{e^2k\alpha^{-3/4}_i}(k\Re{\mathcal M}_{i,j}(\chi))^2\Big ( \alpha^{3/4}_{j+1}|{\mathcal
M}_{j+1, l}(\chi)|\Big)^{2\lceil 1/(10\alpha_{j+1})\rceil }.
\end{split}
 \end{align*}

  As the right-hand side of the expression above is non-negative, we further deduce that
\begin{align}
\label{L2k}
\begin{split}
  \sum_{\chi \in \mathcal{S}(j)} & \left| L \left( \tfrac{1}{2}, \chi \right) \right|^{2k}  \ll  \exp \left( \frac {4k}{\alpha_j} \right)
 \sum^{\mathcal{J}}_{l=j+1} \sum_{\substack{ \chi \shortmod Q }}
\prod^j_{i=1}E_{e^2k\alpha^{-3/4}_i}(k\Re{\mathcal M}_{i,j}(\chi))^2\Big ( \alpha^{3/4}_{j+1}|{\mathcal
M}_{j+1, l}(\chi)|\Big)^{2\lceil 1/(10\alpha_{j+1})\rceil }.  
\end{split}
 \end{align}

 Now, we define functions $c_i(f), 1 \leq i \leq \mathcal{J}$ to be the indicator function of the condition that $f$ is composed of at most $\lceil e^2k\alpha^{-3/4}_i \rceil$ primes, all from the interval $(|Q| ^{\alpha_{i-1}}, |Q| ^{\alpha_{i}}]$. Also, let $c_{j+1}(f)$ the indicator of the condition that if $f$ is composed of exactly $\lceil 1/(10\alpha_{j+1})\rceil$ primes (counted with multiplicity), all from the interval $(|Q| ^{\alpha_{i}}, |Q| ^{\alpha_{i+1}}]$. Furthermore, we define the totally multiplicative function $\beta_j, \gamma_{\chi}$ such that
\begin{align*}
\beta_j(P)= \frac{1}{|P|^{1/\log |Q| ^{\alpha_{j}}}}  \frac{\log (|Q| ^{\alpha_{j}}/|P|)}{\log |Q| ^{\alpha_{j}}} \quad \mbox{and} \quad
 \gamma_{\chi}(P)= \frac {\chi(P)+\overline{\chi(P)}}{2}.
\end{align*}

The above notations, together with those used in Section \ref{sec 4}, allow us to write
\begin{align*}
\begin{split}
 E_{e^2k\alpha^{-3/4}_i}(k\Re {\mathcal M}_{i,j}(\chi))=& \sum_{f_i} \frac{\beta_j(f_i)}{\sqrt{|f_i|}} \frac{k^{\Omega(f_i)}}{w(f_i)}  c_i(f_i) \gamma_{\chi}(f_i), \quad 1 \le i \le j, \\
\Big ( {\mathcal M}_{j+1, l}(\chi)\Big)^{\lceil 1/(10\alpha_{j+1})\rceil } =&  \sum_{ \substack{f_{j+1}}}\frac{\beta_l(f_{j+1})}{\sqrt{f_{j+1}}}\frac{(\lceil 1/(10\alpha_{j+1})\rceil)!
  }{w(f_{j+1})}c_{j+1}(f_{j+1})\chi(f_{j+1}).
\end{split}
\end{align*}

 We apply the above to recast the sum over $\chi$ in \eqref{L2k} as
\begin{align}
\label{sumoverchi}
\begin{split}
& \Big ( \alpha^{3/4}_{j+1}\Big)^{2\lceil 1/(10\alpha_{j+1})\rceil }(\lceil 1/(10\alpha_{j+1})\rceil!)^2\sum_{\substack{f_i, f'_i\\ 1 \leq i \leq j+1}} \frac{\beta_l(f_{j+1}f'_{j+1})\prod^{j}_{i=1}\beta_j(f_if'_i))}{\sqrt{\prod^{j+1}_{i=1}|f_if_i'|}}
\frac{k^{\Omega(\prod^{j+1}_{i=1}f_if_i')}}{w(\prod^{j+1}_{i=1}f_if_i')}\prod^{j+1}_{i=1}c_i(f_if_i') \\
& \times \sum_{\substack{ \chi \shortmod Q }}\chi(f_{j+1})\overline{\chi(f'_{j+1})}\prod^j_{i=1}\gamma_{\chi}(f_if_i').
\end{split}
\end{align}

  We may write the sum over $\chi$ above in the form
\begin{align*}
\begin{split}
  \sum_{f,g}c_{f,g}\sum_{\substack{ \chi \shortmod Q }}\chi(f)\overline{\chi(g)},
\end{split}
\end{align*}
  where $c_{f,g}$ depends on $f,g$ only. Then it is easy to see that
\begin{align*}
\begin{split}
 |f|, |g| \ll 
\Big (\prod^j_{i=1}|Q|^{\alpha_i \cdot e^2k\alpha^{-3/4}_i} \Big ) \cdot |Q|^{\alpha_{i+1} \cdot \lceil 1/(10\alpha_{j+1})\rceil } \ll |Q|^{1-\varepsilon}.
\end{split}
\end{align*}

  It follows from this and the orthogonal relation given in \eqref{orthrel} that only diagonal terms contribute to \eqref{sumoverchi}. More specifically, a typical sum of the form
\begin{align*}
\begin{split}
  \sum_{\substack{ \chi \shortmod Q }}\chi(f_{j+1})\overline{\chi(f'_{j+1})}\prod_{|P|^{l_P}\| f}\Big (\frac {\chi(P)+\overline{\chi(P)}}{2}\Big )^{l_P}.
\end{split}
\end{align*}
  is non-zero if and only if $f_{j+1}=f'_{j+1}$ and each $l_P$ is even, in which case the sum equals to
\begin{align*}
\begin{split}
  \varphi(Q)\prod_{|P|^{l_P}\| f} 2^{-l_P}\binom {l_P}{l_P/2}.
\end{split}
\end{align*}

 We then deduce from the above that
\begin{align}
\label{sumchiest}
\begin{split}
&\sum_{\substack{ \chi \shortmod Q }}
\prod^j_{i=1}E_{e^2k\alpha^{-3/4}_i}(k\Re{\mathcal M}_{i,j}(\chi))^2\Big ( \alpha^{3/4}_{j+1}|{\mathcal
M}_{j+1, l}(\chi)|\Big)^{2\lceil 1/(10\alpha_{j+1})\rceil } \\
\ll & |Q|\Big ( \alpha^{3/4}_{j+1}\Big)^{2\lceil 1/(10\alpha_{j+1})\rceil }\frac {(\lceil 1/(10\alpha_{j+1})\rceil!)^2}{(\lceil 1/(10\alpha_{j+1})\rceil)!}\prod_{|P| \leq |Q| ^{\alpha_j}}I_0\left( \frac {2k \beta_j(P)}{|P|^{1/2}}\right)  \Big ( \sum_{|Q| ^{\alpha_{j}} < |P| \leq |Q| ^{\alpha_{j+1}}} \frac{\beta^2_l(P)}{|P|}\Big )^{\lceil 1/(10\alpha_{j+1})\rceil }.
\end{split}
\end{align}
  where (see \cite[p. 492]{Kirila})
\begin{align*}
\begin{split}
 I_0(z)=\sum^{\infty}_{n=0}\frac {(z/2)^{2n}}{(n!)^2}
\end{split}
 \end{align*}
  is the modified Bessel function of the first kind. \newline

  Note that we have for $1 \leq i \leq \mathcal{J}-1$,
\begin{align*}
\mathcal{J}-i \leq \frac{\log(1/\alpha_{i})}{\log 20}  \quad \mbox{and} \quad \sum_{|Q| ^{\alpha_{i}} < |P| \leq |Q| ^{\alpha_{i+1}}} \frac{1}{|P|}
 = \log \alpha_{i+1} - \log \alpha_{i} + o(1) = \log 20 + o(1) \leq 10 .
\end{align*}

  We apply Lemma \ref{RS}, \eqref{Stirling} and the above to estimate the last expression in \eqref{sumchiest} to see that it is
\begin{align*}
\begin{split}
\ll & |Q| e^{-200k/\alpha_{j+1}}  \prod_{|P| \leq |Q| ^{\alpha_j}}\left( 1+\frac {k^2}{|P|}+O \left( \frac 1{|P|^2} \right) \right) \ll e^{-200k/\alpha_{j+1}}  |Q| (\log_q |Q|)^{k^2}.
\end{split}
 \end{align*}

   We then conclude from the above and \eqref{L2k}, noting that $20/\alpha_{j+1}=1/\alpha_j$, that
\begin{align*}
\begin{split}
  \sum_{\chi \in \mathcal{S}(j)} |L(1/2, \chi)|^{2k}
\ll&  (\mathcal{J}-j) e^{4k/\alpha_j} e^{-200k/\alpha_{j+1}}  |Q| (\log_q |Q|)^{k^2}\ll   e^{-2k/\alpha_{j}}|Q|  (\log_q |Q| )^{k^{2}}.
\end{split}
 \end{align*}

   As the sum of the right-hand side expression over $j$ converges, we see that the above bound implies \eqref{sumovermj}
and this completes the proof of Proposition \ref{thmupperbound}.

\vspace*{.5cm}

\noindent{\bf Acknowledgments.}  P. G. is supported in part by NSFC grant 11871082 and L. Z. by the Faculty Silverstar Grant PS65447 at the University of New South Wales. The authors are very grateful to the anonymous referee for his/her very careful reading of the manuscript and many helpful suggestions. 

\bibliography{biblio}
\bibliographystyle{amsxport}

\end{document}